\documentclass{article}
\usepackage{amsmath,amssymb,amsthm, amscd, verbatim, setspace, breqn,amsrefs}
\usepackage{graphicx} 
\usepackage{comment, cite}
\usepackage{algorithm,algorithmic}

\theoremstyle{plain}
\newtheorem{thm}{Theorem} 
\newtheorem*{thmn}{Theorem} 
\numberwithin{thm}{section}
\newtheorem*{lemn}{Lemma}
\newtheorem{lem}[thm]{Lemma}
\newtheorem{cor}[thm]{Corollary} 
\newtheorem{prop}[thm]{Proposition}

\theoremstyle{definition}
\newtheorem{defn}[thm]{Definition}

\theoremstyle{remark}
\newtheorem{rem}[thm]{Remark}

\newtheorem{ex}{Example}

\title{Nondegenerate $2 \times k \times (k+1)$ Hypermatrices}
\author{Colin Aitken}
\begin{document}
\maketitle
\begin{abstract}
We construct an extension of Gaussian elimination to show that if $\mathbb{F}$ is a topological field, then there is a transitive, free, and continuous action of a natural quotient of $GL_k(\mathbb{F}) \times GL_{k+1}(\mathbb{F})$ on the set $M_k(\mathbb{F})$ of $2 \times k \times (k+1)$ hypermatrices over $\mathbb{F}$ with nonzero hyperdeterminant. 

We use this action to answer a number of questions including determining the homotopy groups of $M_k(\mathbb{C})$, counting elements of $M_k(\mathbb{F}_q)$ (generalizing an unpublished result of Lewis and Sam), and computing hyperdeterminants for $2 \times k \times (k+1)$ hypermatrices in $O(k^4)$ time, which we use to compute explicit formulas in some special cases.

\end{abstract}
\section{Introduction}
A hypermatrix of format $(k_1 + 1) \times (k_2 + 1) \times \cdots \times (k_r + 1)$ over some field $\mathbb{F}$ is an $r$-dimensional array of elements $a_{i_1,\cdots,i_r}$ of $\mathbb{F}$ with $0 \leq i_j \leq k_j$. This can be viewed as an element of the tensor product \[ \mathbb{F}^{(k_1 + 1)} \otimes \cdots \otimes \mathbb{F}^{(k_r + 1)},\] and as such comes with a natural action of the  product of the linear groups $GL_{k_1 + 1}(\mathbb{F}) \times \cdots \times GL_{k_r + 1}(\mathbb{F})$. Whereas the $GL$ actions for matrices correspond to row operations, they here correspond to \emph{slice} operations, where a slice is an $(r-1)$ dimensional subarray of the hypermatrix.

The hyperdeterminant was originally defined by Cayley in \cite{Cayley} and rediscovered by Gelfand, Kapranov, and Zelevinsky in \cite{GKZ}.
\begin{defn}The hyperdeterminant $\operatorname{Det}$ of format $(k_1 + 1) \times \cdots \times (k_r + 1)$ is the unique irreducible relatively $GL$-invariant\footnote{By ``relatively $GL$-invariant'', we mean that for any $1 \leq i \leq r$, there is an integer $l_i$ such that for any element $g \in GL_{k_i + 1}$ and $(k_1 + 1) \times \cdots \times (k_r + 1)$ hypermatrix $M$, we have $\operatorname{Det}(g \cdot M) = \det(g)^{l_i}\operatorname{Det}(M)$.}  polynomial which is zero if and only if there is a solution over $\overline{\mathbb{F}}$ to
\[
f(x) = \dfrac{\partial f(x)}{\partial x_i^{(j)}} = 0 
\]
with $x^{(j)} \neq 0$ where $f$ is the multilinear form defined by:
\[
f(x^{(1)},x^{(2)},\cdots,x^{(r)}) = \sum_{i_1,\cdots,i_r} a_{i_1,\cdots,i_r}x_{i_1}^{(1)}\cdots x_{i_r}^{(r)}
\]
for the hypermatrix $a_{i_1,\cdots,i_r}$. A hypermatrix with zero hyperdeterminant is called \emph{degenerate.}
\end{defn}
In general hyperdeterminants are extremely difficult to study---in particular, it was shown by Hillar and Lim in \cite{Tensors} that even the question of whether a given hypermatrix has nonzero hyperdeterminant is NP-Hard. However, the \emph{boundary format} in which $k_r = k_1 + k_2 + \cdots + k_{r-1}$ is shown in \cite{GKZ} to be much simpler. In particular, a boundary format hypermatrix is nondegenerate if and only if there is a nontrivial solution over $\overline{\mathbb{F}}$ to
\[
f_0(x) = f_1(x) = \cdots = f_{k_r}(x) = 0,
\]
where
\[
f_{i_r} = \sum_{i_1,\cdots,i_{r-1}} a_{i_1,i_2,\cdots,i_r}x_{i_1}^{(1)} \cdots x_{i_r}^{(r)}.
\]

We can represent a $2 \times k \times (k+1)$ hypermatrix as two $k \times (k+1)$ matrices. In this case, the $GL_k$ and $GL_{k+1}$ actions act by simultaneous row and column operations on the two matrices.  In this paper, we will first prove that all $2 \times k \times (k+1)$ nondegenerate hypermatrices fall into a single orbit under the $GL_k(\mathbb{F}) \times GL_{k+1}(\mathbb{F})$ action. This is perhaps surprising in light of Belitskii and Sergeichukk's Theorem $4.4$ of \cite{BS}, which implies that there are infinitely many $GL_2(\mathbb{C}) \times GL_k(\mathbb{C}) \times GL_{k+1}(\mathbb{C})$ orbits of $\mathbb{C}^2 \otimes \mathbb{C}^k \otimes \mathbb{C}^{k+1}$ for $k \geq 4$. In the remaining sections, we will use this theorem to determine the number of such hypermatrices over finite fields, understand the topology of spaces of nondegenerate hypermatrices over $\mathbb{R}$ and $\mathbb{C}$, and compute explicit formulas for hyperdeterminants.

\section{Main Theorem}
We begin by introducing the set of nondegenerate hypermatrices, and the group with which we would like to act on them:
\begin{defn}
Let $M_k(\mathbb{F})$ be the set of all nondegenerate $2 \times k \times (k+1)$ hypermatrices over the field $\mathbb{F}$, and let 
\[
G = GL_k(\mathbb{F}) \times GL_{k+1}(\mathbb{F})/N
\]
where $N$ is the subgroup of $GL_k \times GL_{k+1}$ consisting of ordered pairs $(cI_k,c^{-1}I_{k+1})$ for $c \in \mathbb{F}^\times.$
\end{defn}
We take a quotient of the product of the $GL$'s rather than the $GL$'s themselves in order to guarantee that the action of $G$ on $M_k$ is free. With this in mind, the goal of this section is to prove the following theorem:
\begin{thm}\label{main}
Let $\mathbb{F}$ be a topological field. Then, there is a continuous, free, and transitive action of $G(\mathbb{F})$ on $M_k(\mathbb{F})$
\end{thm}

This action is induced by the action of $GL_k \times GL_{k+1}$, which also implies its continuity. This means we only need to check that the action is free and transitive. To show that the action is transitive, we will introduce a reduction algorithm using elements of $G(\mathbb{F})$ to reduce every arbitrary nondegenerate hypermatrix to a single hypermatrix:

We will prove Theorem \ref{main} using a series of lemmas. First we will construct a slightly different way of checking nondegeneracy of a hypermatrix, which we will need to show that Algorithm $1$ correctly identifies degenerate hypermatrices.
\begin{lem}
\label{degen}
Let $M$ be a $(k_1 + 1) \times (k_2 + 1) \times (k_1 + k_2 + 1)$ hypermatrix over field $\mathbb{F}$, and denote the $(k_2 + 1) \times (k_1 + k_2 +1)$ slices of $M$ by $M_0, M_1, \cdots, M_{k_1}.$ Then, $M$ is nondegenerate if and only if every linear combination $c_0M_0 + \cdots + c_{k_1 }M_{k_1}$ of the $M_i$'s over $\overline{\mathbb{F}}$ with $c_0,\cdots,c_{k_1}$ not all zero has full rank.
\end{lem}
\begin{proof}
We recall from \cite{DO} the notion of multiplication of a hypermatrix by a vector, by which we mean the operation of taking linear combinations of slices with coefficients indexed by the vector. For example, when multiplying a $2 \times 3 \times 4$ hypermatrix by a $3$-vector, the result would be  a $2 \times 4$ matrix.

Since $M$ is a boundary format hypermatrix, it is degenerate if and only if there is a nonzero solution to \[
f_0(x) = f_1(x) = \cdots = f_{k_1 + k_2 + 1}(x) = 0
\]

This is equivalent to the existence of a pair of vectors $(v,w) \in \overline{\mathbb{F}}^{k_1 + 1} \times \overline{\mathbb{F}}^{k_2 + 1}$ such that $(Mv)w = 0$, which means that $(Mv)$ has less than full rank. But if $v = (v_1,\cdots,v_{k_1+1})$, then $(Mv)$ is the same as $\sum M_iv_i$, which proves the lemma.
\end{proof}

Next, we show that we can use the $GL_k \times GL_{k+1}$ action to reduce each nondegenerate hypermatrix to a standard form, which will show that the action is transitive. Before presenting the algorithm in full, we will look at the following toy example.

\begin{ex}
Consider the $2 \times 3 \times 4$ hypermatrix
\[
\begin{pmatrix}1 & 0 & 0 \\
0 & 1 & 0 \\
0 & 0 & 1 \\
0 & 0 & 0 \\
\end{pmatrix}
\begin{pmatrix}
1 & 0 & 0 \\
1 & 0 & 0 \\
-3 & 3 & 0 \\
1 & 2 & 1\\
\end{pmatrix}.
\]
We would like to reduce it to the form
\[
\begin{pmatrix}
1 & 0 & 0 \\
0 & 1 & 0 \\
0 & 0 & 1 \\
0 & 0 & 0 \\
\end{pmatrix}
\begin{pmatrix}
0 & 0 & 0 \\
1 & 0 & 0 \\
0 & 1 & 0 \\
0 & 0 & 1\\
\end{pmatrix}
\]
using row and column operations. In doing so, we will exhibit the four basic pieces of our reduction algorithm.
\begin{enumerate}
\item \emph{Clearing nonzero elements of a row.}

The last row of the second slice has a $1$ and a $2$ where there should be zeroes, so we will eliminate them by using the  $GL_3$ action to add multiples of the third column:
\[
\begin{pmatrix}1 & 0 & 0 \\
0 & 1 & 0 \\
-1 & -2 & 1 \\
0 & 0 & 0 \\
\end{pmatrix}
\begin{pmatrix}
0 & 1 & 0 \\
0 & 1 & 0 \\
3 & -3 & 0 \\
0 & 0 & 1\\
\end{pmatrix}.
\]
To ``fix" the problem this created in the first slice, we will use the $GL_4$ action to add multiples of the first and second row to the third row:
\[
\begin{pmatrix}1 & 0 & 0 \\
0 & 1 & 0 \\
0 & 0 & 1 \\
0 & 0 & 0 \\
\end{pmatrix}
\begin{pmatrix}
0 & 1 & 0 \\
0 & 1 & 0 \\
3 & 0 & 0 \\
0 & 0 & 1\\
\end{pmatrix}.
\]
In general, this will only affect elements above the row being cleared.
\item \emph{Making diagonal elements nonzero.}

In the second slice there is a zero on the diagonal where there needs to be a $1.$ We will use the $GL_3$ action to swap the first and second columns:
\[
\begin{pmatrix}0 & 1 & 0 \\
1 & 0 & 0 \\
0 & 0 & 1 \\
0 & 0 & 0 \\
\end{pmatrix}
\begin{pmatrix}
1 & 0 & 0 \\
1 & 0 & 0 \\
0 & 3 & 0 \\
0 & 0 & 1\\
\end{pmatrix}.
\]
To ``fix'' the problem this created in the first slice, we will use the $GL_4$ action to swap the first and second rows:
\[
\begin{pmatrix}1 & 0 & 0 \\
0 & 1 & 0 \\
0 & 0 & 1 \\
0 & 0 & 0 \\
\end{pmatrix}
\begin{pmatrix}
1 & 0 & 0 \\
1 & 0 & 0 \\
0 & 3 & 0 \\
0 & 0 & 1\\
\end{pmatrix}.
\]
In general, this only affects elements above the row being cleared.

\item \emph{Making diagonal elements $1$.}

In the second slice there is a $3$ on the diagonal where there needs to be a $1$. We will use the $GL_4$ action to divide the third row by three:
\[
\begin{pmatrix}1 & 0 & 0 \\
0 & 1 & 0 \\
0 & 0 & 1/3 \\
0 & 0 & 0 \\
\end{pmatrix}
\begin{pmatrix}
1 & 0 & 0 \\
1 & 0 & 0 \\
0 & 1 & 0 \\
0 & 0 & 1\\
\end{pmatrix}.
\]
To ``fix'' the $1/3$ this created in the first slice, we will use the $GL_3$ action to multiply the third column by three:
\[
\begin{pmatrix}1 & 0 & 0 \\
0 & 1 & 0 \\
0 & 0 & 1 \\
0 & 0 & 0 \\
\end{pmatrix}
\begin{pmatrix}
1 & 0 & 0 \\
1 & 0 & 0 \\
0 & 1 & 0 \\
0 & 0 & 3\\
\end{pmatrix}.
\]
Finally, we will use the $GL_4$ action divide the fourth row by three. This will not change in the first slice.
\[
\begin{pmatrix}1 & 0 & 0 \\
0 & 1 & 0 \\
0 & 0 & 1 \\
0 & 0 & 0 \\
\end{pmatrix}
\begin{pmatrix}
1 & 0 & 0 \\
1 & 0 & 0 \\
0 & 1 & 0 \\
0 & 0 & 1\\
\end{pmatrix}.
\]
In general, we can follow this pattern of chasing elements down and to the right to make the diagonal elements $1.$
\item \emph{Clearing nonzero elements of a column.}
The leftmost column of the second slice has an extra $1$, which we can clear by using the $GL_4$ action to subtract the second row from the first.
\[
\begin{pmatrix}1 & -1 & 0 \\
0 & 1 & 0 \\
0 & 0 & 1 \\
0 & 0 & 0 \\
\end{pmatrix}
\begin{pmatrix}
0 & 0 & 0 \\
1 & 0 & 0 \\
0 & 1 & 0 \\
0 & 0 & 1\\
\end{pmatrix}.
\]
To ``fix'' the $-1$ this created in the first slice, we use the $GL_3$ action to add the first column to the second.
\[
\begin{pmatrix}1 & 0 & 0 \\
0 & 1 & 0 \\
0 & 0 & 1 \\
0 & 0 & 0 \\
\end{pmatrix}
\begin{pmatrix}
0 & 0 & 0 \\
1 & 1 & 0 \\
0 & 1 & 0 \\
0 & 0 & 1\\
\end{pmatrix}.
\]
We can continue this pattern to ``chase'' the $1$ down the diagonal until it goes away.
\[
\begin{pmatrix}1 & 0 & 0 \\
0 & 1 & -1 \\
0 & 0 & 1 \\
0 & 0 & 0 \\
\end{pmatrix}
\begin{pmatrix}
0 & 0 & 0 \\
1 & 0 & 0 \\
0 & 1 & 0 \\
0 & 0 & 1\\
\end{pmatrix}
\]
\[
\begin{pmatrix}1 & 0 & 0 \\
0 & 1 & 0 \\
0 & 0 & 1 \\
0 & 0 & 0 \\
\end{pmatrix}
\begin{pmatrix}
0 & 0 & 0 \\
1 & 0 & 0 \\
0 & 1 & 1 \\
0 & 0 & 1\\
\end{pmatrix}
\]
\[
\begin{pmatrix}1 & 0 & 0 \\
0 & 1 & 0 \\
0 & 0 & 1 \\
0 & 0 & 0 \\
\end{pmatrix}
\begin{pmatrix}
0& 0 & 0 \\
1 & 0 & 0 \\
0 & 1 & 0 \\
0 & 0 & 1\\
\end{pmatrix}
\]
\end{enumerate}
\end{ex}
We can extend these four basic moves into an algorithm. To start, we reduce the first slice to the specified form by standard Gaussian elimination. We then begin at the bottom right corner of the second slice and use the second move to make sure it is nonzero, followed by the third move to make sure that it's one. We can then use the first move to clear its row and the last move to clear its column, and then move to the next diagonal element. 

More formally, we get the following algorithm to reduce the second slice:
\begin{algorithm}\caption{Double Gaussian Elimination for $2 \times k \times (k+1)$ Hypermatrices}
\label{alg:model}
\noindent\begin{algorithmic}
\renewcommand\algorithmicdo{}
\renewcommand\algorithmicthen{}
\STATE $j \leftarrow k-1$
\WHILE {$j \geq 0$}
\IF {$a_{1(j+1)j} = 0$}
\FOR {$l \in \{0,1,\cdots,j-1\}$ \hfill \emph{Make diagonal elements nonzero}}
\IF {$a_{1(j+1)l} \neq 0$}
\STATE Swap columns $l$ and $j$
\STATE Swap rows  $l$ and $j$
\STATE \textbf{break} (out of the \emph{for} loop.)
\ENDIF
\ENDFOR
\IF {$a_{1(j+1)j} = 0$}
\STATE Error: ``Hypermatrix is Degenerate''
\ENDIF
\ENDIF
\STATE $c \leftarrow a_{1(j+1)j}$ \hfill \emph{Make diagonal elements $1$}
\STATE Multiply rows $j+1, j+2 \cdots, k$ by $1/c.$ 
\STATE Multiply columns $j+1, j+2, \cdots k-1$ by $c.$ 
\FOR {$\ell \in \{0,1,2,\cdots,j-1\}$ \hfill \emph{Clear the rest of the row}}
\STATE $c \leftarrow a_{1(j+1)\ell}$ 
\STATE Add $-c$ times column $j$ to column $\ell.$
\STATE Add $c$ times row $\ell$ to row $j$.
\ENDFOR
\FOR {$m \in \{j+1,j+2,\cdots,k\}$ \hfill \emph{Clear the rest of the column}}
\FOR {$\ell \in \{0,1,2,\cdots,m-1\}$}
\STATE $c \leftarrow a_{1\ell (m-1)}$
\STATE Add $-c$ times row $m$ to row $\ell$.
\IF {$m < k$}
\STATE Add $c$ times column $\ell$ to column $m$
\ENDIF
\ENDFOR
\ENDFOR
\STATE $j \leftarrow (j-1)$
\ENDWHILE
\end{algorithmic}
\end{algorithm}
\begin{lem}\label{mainlem}
Let $M$ be a nondegenerate $2 \times k \times (k+1)$ hypermatrix whose first $k \times (k+1)$ slice is of the form:
\[
\begin{pmatrix}
1 & 0 & \cdots & 0 \\
0 & 1 & \cdots & 0 \\
\vdots & \vdots & \ddots & \vdots \\
0 & 0 & \cdots & 1 \\
0 & 0 & \cdots & 0
\end{pmatrix}
\]
Then, applying Algorithm $1$ to $M$ will result in the hypermatrix:
\[I_{k,k+1} := 
\begin{pmatrix}
1 & 0 & \cdots & 0 \\
0 & 1 & \cdots & 0 \\
\vdots & \vdots & \ddots & \vdots \\
0 & 0 & \cdots & 1 \\
0 & 0 & \cdots & 0
\end{pmatrix}
\begin{pmatrix}
0 & 0 & \cdots & 0 \\
1 & 0 & \cdots & 0 \\
0 & 1 & \cdots & 0 \\
\vdots & \vdots & \ddots & \vdots \\
0 & 0 & \cdots & 1
\end{pmatrix}
\]
\end{lem}
\begin{proof}
That the algorithm  does indeed reduce a hypermatrix to $I_{k,k+1}$ if it does not throw an error is clear by working through the steps of the algorithm and noting that each run through the outermost \emph{while} loop fixes one row and one column of the second slice without affecting the first slice or rows and columns which have already been fixed.

The only part we need to prove is that if a row of all zeroes is encountered before reaching the top, then the hypermatrix is degenerate. We proceed by contradiction. Suppose at some point in the algorithm we come across a row of all zeroes, say row $j+1$. This implies that each of the two large slices split as the direct sum of a $2 \times j \times j$ hypermatrix whose first slice is an identity matrix, and a $2 \times (k-j) \times (k+1 - j)$ hypermatrix. Let $A$ be the second slice of the $2 \times j \times j$ hypermatrix, and let $\lambda$ be an eigenvalue of $A$. Then $(\lambda(I) - A)$ has less than full rank, so by Lemma \ref{degen} the hypermatrix is indeed degenerate.
\end{proof}
Finally, we are in a position to prove the main theorem of this section.
\begin{proof}[Proof of Theorem \ref{main}]
Lemma \ref{mainlem} implies that the action is transitive because we can use row operations to reduce the first slice, then Algorithm $1$ to reduce the second slice, which implies that all elements of $M_k(\mathbb{F})$ lie in the same orbit as $I_{k,k+1}$ and therefore the same orbit as each other. Therefore, it only remains to check that the group action is free. Let $A \in GL_{k}$ and $B \in GL_{k+1}$. We want to show that if $(A,B) \cdot I_{k,k+1} = I_{k,k+1}$, then $A = cI_k$ and $B = c^{-1}I_{k+1}$ for some $c.$ 

Let $A = (a_{ij})$ and $B = (b_{ij})$ such that $(A,B) \cdot I_{k,k+1} = I_{k,k+1}.$ Expanding the first slice implies that $a_{k+1,1} = a_{k+1,2} = \cdots = a_{k+1,k} = 0$ and that
\[
\begin{pmatrix}
a_{11} & a_{12} & \cdots & a_{1k} \\
a_{21} & a_{22} & \cdots & a_{2k} \\
\vdots & \vdots & \ddots & \vdots \\
a_{k1} & a_{k2} & \cdots & a_{kk}
\end{pmatrix} = B^{-1},
\]while expanding the second slice implies that $a_{12} = a_{13} = \cdots = a_{1,k+1} = 0$ and that 
\[
\begin{pmatrix}
a_{22} & a_{23} & \cdots & a_{2,k+1} \\
a_{32} & a_{33} & \cdots & a_{3,k+1} \\
\vdots & \vdots & \ddots & \vdots \\
a_{k+1,2} & a_{k+1,3} & \cdots & a_{k+1,k+1}
\end{pmatrix} = B^{-1}
\]
This implies that
\[
\begin{pmatrix}
a_{22} & a_{23} & \cdots & a_{2,k+1} \\
a_{32} & a_{33} & \cdots & a_{3,k+1} \\
\vdots & \vdots & \ddots & \vdots \\
0 & 0 & \cdots & a_{k+1,k+1}
\end{pmatrix} = \begin{pmatrix}
a_{11} & 0 & \cdots &0 \\
a_{21} & a_{22} & \cdots & a_{2k} \\
\vdots & \vdots & \ddots & \vdots \\
a_{k1} & a_{k2} & \cdots & a_{kk}
\end{pmatrix},
\]
and so $a_{11} = a_{22} = \cdots a_{k+1,k+1}$. Looking at the top rows implies that $a_{23} = \cdots = a_{2,{k+1}} = 0$, which implies $a_{34} = \cdots = a_{3,k+1} = 0$, and so on. Continuing via induction, we see that $a_{ij} = 0$ for $i \neq j.$ Therefore $A = a_{11}I_{k+1}$, so $B = a_{11}^{-1}I_k$, so $(A,B)$ lies in the subgroup $N$ that was quotiented out and the action is indeed free. 
\end{proof}
\begin{cor} \label{homeo}
If $\mathbb{F}$ is Hausdorff, then there exists a homeomorphism $\phi: G \to M_k(\mathbb{F})$.
\end{cor}
\begin{proof}
Pick an arbitrary $x \in M_k(\mathbb{F})$, and let $\phi(g) = g\cdot x$. The above theorem implies that $\phi$ is a continuous bijection, and the fact that $M_k(\mathbb{F})$ is Hausdorff implies $\phi$ is a homeomorphism. 
\end{proof}
\section{Some Consequences}
\subsection{Counting}
We will denote by $[n]_q$ the sum:
\[
[n]_q = 1 + q + \cdots + q^{n-1}
\]
and by $[n]!_q$ the product:
\[
[n]!_q = [1]_q[2]_q\cdots[n]_q.
\]
Then, the formula
\[
|GL_n(\mathbb{F}_q)| = q^{\binom{n}{2}}(q-1)^n[n]!_q
\]
is well-known, and indeed motivates viewing invertible matrices as a $q$-analogue of permutations, as in Section $1.10$ of \cite{Stanley}. This leads one to consider other sets of nondegenerate hypermatrices and their sizes over finite fields. Nondegenerate hypermatrices of format $2 \times 2 \times 2$ over $\mathbb{F}_q$ have been counted in an unpublished manuscript of Musiker and Yu\cite{myu}.
\begin{thmn}[Musiker-Yu]
The number of nondegenerate $2 \times 2 \times 2$ hypermatrices over $\mathbb{F}_q$ is $q^3(q-1)^2[4]_q$
\end{thmn}
The $2 \times 2 \times 3$ case was solved in unpublished work of Joel Lewis and Steven Sam (personal communication).
\begin{thmn}[Lewis-Sam]
The number of nondegenerate $2 \times 2 \times 3$ hypermatrices over $\mathbb{F}_q$ is $q^4(q-1)^4[2]^2_q[3]_q$.
\end{thmn}
Viewing $\mathbb{F}_q$ as a topological field with the discrete topology, Corollary \ref{homeo} above allows us to answer this question for general $2 \times k \times (k+1)$ hypermatrices, generalizing Lewis and Sam's result.
\begin{prop}
The number of nondegenerate $2 \times k \times (k+1)$ hypermatrices over $\mathbb{F}_q$ is $q^{k^2}(q-1)^{2k}[k]!_q[k+1]!_q$
\end{prop}
\begin{proof}
By Corollary \ref{homeo}, there is a bijection $M_k(\mathbb{F}_q) \to G$. This implies:
\begin{align*}
|M_k(\mathbb{F}_q)| &= |G| \\
                    &= \dfrac{|GL_k|\cdot|GL_{k+1}|}{|\mathbb{F}_q^\times|}\\
										&= q^{k^2}(q-1)^{2k}[k]!_q[k+1]!_q.
\end{align*}

\end{proof}
\subsection{Topology}
In this section, we consider the topology of $M_k(\mathbb{F})$, where $\mathbb{F}$ is either $\mathbb{R}$ or $\mathbb{C}.$ The topological information is mostly contained in the following fiber bundle.
\begin{cor}
For $\mathbb{F} = \mathbb{R},\mathbb{C}$, $M_k(\mathbb{F})$ is the base space of a fiber bundle:
\[
\mathbb{F}^\times \to GL_k(\mathbb{F}) \times GL_{k+1}(\mathbb{F}) \to M_k(\mathbb{F})
\]
\end{cor}
\begin{proof}
By Corollary \ref{homeo}, $M_k(\mathbb{F}) \approx G$. The corresponding fiber bundle for $G$ comes from the exact sequence of Lie groups
\[
0 \to \langle(cI_k,c^{-1}I_{k+1})_{c \in \mathbb{F}^\times} \rangle \to GL_k(\mathbb{F}) \times GL_{k+1}(\mathbb{F}) \to G \to 0
\]
where the fiber is clearly closed and homeomorphic to $\mathbb{F}^\times$. 
\end{proof}
This directly gives us the homotopy groups of $M_k(\mathbb{C})$. 
\begin{cor}
 $M_k(\mathbb{C})$ is connected, and has homotopy groups as follows:
\[
\pi_n(M_k(\mathbb{C})) = \begin{cases} \mathbb{Z} &\text{if } n = 1 \\
\pi_n(GL_k(\mathbb{C})) \times \pi_n(GL_{k+1}(\mathbb{C})) &\text{if } n \geq 2											
													 \end{cases}
\]
\end{cor}
\begin{proof}
This follows directly from the long exact sequence of a fibration applied to the above fiber bundle. In the $\pi_1$ case we obtain the exact sequence
\[
\pi_1(\mathbb{C}^\times) \to \pi_1(GL_k(\mathbb{C}) \times GL_{k+1}(\mathbb{C})) \to \pi_1(M_k(\mathbb{C})) \to 0
\]
which becomes
\[
\mathbb{Z} \to \mathbb{Z} \times \mathbb{Z} \to \pi_1(M_k(\mathbb{C})) \to 0
\]
where the first map takes $1$ to $(1,1)$. This implies $\pi_1(M_k(\mathbb{C})) = \mathbb{Z}.$ 	

Since $\pi_n(\mathbb{C}^\times) = 0$ for $n>1$, we obtain:
\[
0 \to \pi_n(GL_k(\mathbb{C}) \times GL_{k+1}(\mathbb{C})) \to \pi_n(M_k(\mathbb{C})) \to 0,
\]
which implies that $\pi_n(M_k(\mathbb{C})) = \pi_n(GL_k(\mathbb{C})) \times \pi_n(GL_{k+1}(\mathbb{C}))$. 
\end{proof}
Over $\mathbb{R}$ we can explicitly determine the homotopy type of $M_k$.
\begin{cor}
$M_k(\mathbb{R})$ is homotopy equivalent to two copies of ${SO(k) \times SO(k+1)}$
\end{cor}
\begin{proof}
Exactly one of $k, k+1$ is odd. We will assume $k$ is odd; the proof of the other case is similar.

We note that $M_k(\mathbb{R})$ deformation retracts onto the space of hypermatrices with hyperdeterminant $\pm 1$. Since the hyperdeterminant is a polynomial in the entries of the hypermatrix and therefore continuous, the set with hyperdeterminant $1$ and the set with hyperdeterminant $-1$ are separated and moreover homeomorphic.

Now, consider the set with hyperdeterminant $1$. As a subspace of $G$, this is the space of pairs $(x,y) \in GL_k(\mathbb{R}) \times GL_{k+1}(\mathbb{R})$ such that\footnote{See the theorem of Dionisi-Ottaviani in the next section} $\det(x)^{k+1}\det(y)^{k} = 1,$ modulo the equivalence relation of multiplying $x$ and $y$ by $c$ and $c^{-1}$ for some $c \in \mathbb{R}$. Each equivalence class has a unique member with $\det(x) = 1$, which implies that $\det(y) = 1$ as well. But this is just $SL_k(\mathbb{R}) \times SL_{k+1}(\mathbb{R}),$ which is homotopy equivalent to $SO(k) \times SO(k+1)$, as desired. 
\end{proof}
\subsection{Explicit Hyperdeterminant Formulas}
A fair amount of recent research is focused on explicitly computing hyperdeterminants--see, for example, \cite{Bremner}. A general formula in terms of determinants of larger matrix can be found in Theorem 14.3.7 of \cite{GKZ}. In the $2 \times k \times (k+1)$ case, this matrix is square of order $k^2 - k.$ Using fast matrix multiplication, this can be computed in $O(k^{4.746})$ time and $O(k^4)$ space, although there may be more efficient methods available due to the sparsity of the matrices. Here, we provide an algorithm which requires $O(k^4)$ time and $O(k^2)$ space.  

We make use of the following lemma and theorem\footnote{Each of these is true in more generality, but we only require the special cases given here.}:
\begin{lemn}[Gelfand-Kapranov-Zelevinsky  \cite{GKZ}, Lemma 14.3.4]
$\operatorname{Det}(I_{k,k+1}) = 1$. 
\end{lemn}
\begin{thmn}[Dionisi-Ottaviani \cite{DO}]\label{DOT}
Let $A \in GL_{k}, B \in GL_{k+1},$ and $M \in M_k$. Then
\[
\operatorname{Det}( (A,B) \cdot M) = \det(A)^{k+1}\det(B)^{k}\operatorname{Det}(M)
\]
\end{thmn}
Since any nondegenerate $2 \times 2 \times (k+1)$ hypermatrix can be reduced to $I_{k,k+1}$ using the $GL_k \times GL_{k+1}$ action, these suffice to compute the hyperdeterminant of an arbitrary $2 \times k \times (k+1)$ hypermatrix:
\begin{prop}
Let $M$ be a $2 \times k \times (k+1)$ hypermatrix over some field $\mathbb{F}$. Then the hyperdeterminant of $M$ can be computed in $O(k^4)$ arithmetic operations.
\end{prop}
\begin{proof}
The previous lemma and theorem imply that if $M = (x,y)I_{k,k+1}
$ for $(x,y) \in GL_k \times GL_{k+1}$, then $\operatorname{Det}(M) = \det(x)^{k+1}\det(y)^k.$ Using Algorithm $1$ and keeping track of the determinants of each row or column operation performed suffices to compute $\operatorname{Det}(M)$.
\end{proof}
In theory, this gives a method for finding the explicit form of the hyperdeterminant: one simply needs to perform Gaussian elimination on a hypermatrix of indeterminates and record the hyperdeterminant. In practice this is a little more difficult, because the rational functions arising as intermediate terms can  become quite large. Nevertheless, we obtain the following two cases:
\begin{thm}[Bremner]
The hyperdeterminant of a $2 \times 2 \times 3$ hypermatrix is the polynomial given in \cite{Bremner}.
\end{thm}

\begin{thm}
Let $M$ be the $2 \times 3 \times 4$ hypermatrix given by
\[
\begin{pmatrix}
1 & 0 & 0 \\
0 & 1 & 0 \\
0 & 0 & 1 \\
0 & 0 & 0 \\
\end{pmatrix}
\begin{pmatrix}
b_{00} & b_{01} & b_{02} \\
b_{10} & b_{11} & b_{12} \\
b_{20} & b_{21} & b_{22} \\
b_{30} & b_{31} & b_{32}  
\end{pmatrix}.
\]
Then,
\begin{align*}
\operatorname{Det}(M) &= b_{32}b_{31}b_{30}b_{22}b_{21}b_{12} - b_{32}^2b_{30}b_{21}^2b_{12} - b_{32}b_{31}^2b_{22}b_{20}b_{12} + b_{32}^2b_{31}b_{21}b_{20}b_{12} \\ &+ b_{31}^2b_{30}b_{21}b_{12}^2 - b_{31}^3b_{20}b_{12}^2 - b_{32}b_{31}b_{30}b_{22}^2b_{11}  + b_{32}^2b_{30}b_{22}b_{21}b_{11} \\&+ b_{32}^2b_{31}b_{22}b_{20}b_{11} - b_{32}^3b_{21}b_{20}b_{11} - b_{31}^2b_{30}b_{22}b_{12}b_{11} - b_{32}b_{31}b_{30}b_{21}b_{12}b_{11} \\&+ 2b_{32}b_{31}^2b_{20}b_{12}b_{11} + b_{32}b_{31}b_{30}b_{22}b_{11}^2 - b_{32}^2b_{31}b_{20}b_{11}^2 + b_{32}b_{31}^2b_{22}^2b_{10} \\&- 2b_{32}^2b_{31}b_{22}b_{21}b_{10} + b_{32}^3b_{21}^2b_{10} + b_{31}^3b_{22}b_{12}b_{10} - b_{32}b_{31}^2b_{21}b_{12}b_{10} \\&- b_{32}b_{31}^2b_{22}b_{11}b_{10} + b_{32}^2b_{31}b_{21}b_{11}b_{10} + b_{32}b_{30}^2b_{22}b_{21}b_{02} - b_{32}b_{31}b_{30}b_{22}b_{20}b_{02} \\&- b_{32}^2b_{30}b_{21}b_{20}b_{02} + b_{32}^2b_{31}b_{20}^2b_{02} + 2b_{31}b_{30}^2b_{21}b_{12}b_{02} - 2b_{31}^2b_{30}b_{20}b_{12}b_{02} \\&- b_{31}b_{30}^2b_{22}b_{11}b_{02} + b_{32}b_{30}^2b_{21}b_{11}b_{02} + b_{31}b_{30}^2b_{11}^2b_{02} + b_{31}^2b_{30}b_{22}b_{10}b_{02} \\&- 3b_{32}b_{31}b_{30}b_{21}b_{10}b_{02} + 2b_{32}b_{31}^2b_{20}b_{10}b_{02} - 2b_{31}^2b_{30}b_{11}b_{10}b_{02} + b_{31}^3b_{10}^2b_{02} \\&+ b_{30}^3b_{21}b_{02}^2 - b_{31}b_{30}^2b_{20}b_{02}^2 - b_{32}b_{30}^2b_{22}^2b_{01} + 2b_{32}^2b_{30}b_{22}b_{20}b_{01} \\&- b_{32}^3b_{20}^2b_{01} - b_{31}b_{30}^2b_{22}b_{12}b_{01} - 2b_{32}b_{30}^2b_{21}b_{12}b_{01} + 3b_{32}b_{31}b_{30}b_{20}b_{12}b_{01} \\&+ b_{32}b_{30}^2b_{22}b_{11}b_{01} - b_{32}^2b_{30}b_{20}b_{11}b_{01} - b_{31}b_{30}^2b_{12}b_{11}b_{01} + 2b_{32}^2b_{30}b_{21}b_{10}b_{01} \\&- 2b_{32}^2b_{31}b_{20}b_{10}b_{01} + b_{31}^2b_{30}b_{12}b_{10}b_{01} + b_{32}b_{31}b_{30}b_{11}b_{10}b_{01} - b_{32}b_{31}^2b_{10}^2b_{01} \\&- b_{30}^3b_{22}b_{02}b_{01} + b_{32}b_{30}^2b_{20}b_{02}b_{01} + b_{30}^3b_{11}b_{02}b_{01} - b_{31}b_{30}^2b_{10}b_{02}b_{01} \\&- b_{30}^3b_{12}b_{01}^2 + b_{32}b_{30}^2b_{10}b_{01}^2 + b_{32}b_{31}b_{30}b_{22}^2b_{00} - b_{32}^2b_{30}b_{22}b_{21}b_{00} \\&- b_{32}^2b_{31}b_{22}b_{20}b_{00} + b_{32}^3b_{21}b_{20}b_{00} + b_{31}^2b_{30}b_{22}b_{12}b_{00} - b_{32}b_{31}^2b_{20}b_{12}b_{00} \\&- b_{32}^2b_{30}b_{21}b_{11}b_{00} + b_{32}^2b_{31}b_{20}b_{11}b_{00} + b_{31}^2b_{30}b_{12}b_{11}b_{00} - b_{32}b_{31}b_{30}b_{11}^2b_{00} \\&- b_{32}b_{31}^2b_{22}b_{10}b_{00} + b_{32}^2b_{31}b_{21}b_{10}b_{00} - b_{31}^3b_{12}b_{10}b_{00} + b_{32}b_{31}^2b_{11}b_{10}b_{00} \\&+ b_{31}b_{30}^2b_{22}b_{02}b_{00} - 2b_{32}b_{30}^2b_{21}b_{02}b_{00} + b_{32}b_{31}b_{30}b_{20}b_{02}b_{00} - b_{31}b_{30}^2b_{11}b_{02}b_{00} \\&+ b_{31}^2b_{30}b_{10}b_{02}b_{00} + b_{32}b_{30}^2b_{22}b_{01}b_{00} - b_{32}^2b_{30}b_{20}b_{01}b_{00} + 2b_{31}b_{30}^2b_{12}b_{01}b_{00} \\&- b_{32}b_{30}^2b_{11}b_{01}b_{00} - b_{32}b_{31}b_{30}b_{10}b_{01}b_{00} - b_{32}b_{31}b_{30}b_{22}b_{00}^2 + b_{32}^2b_{30}b_{21}b_{00}^2 \\&- b_{31}^2b_{30}b_{12}b_{00}^2 + b_{32}b_{31}b_{30}b_{11}b_{00}^2
\end{align*}
\end{thm}
\begin{proof}
Algorithm $1$, implemented in Sage\cite{sage}. 
\end{proof}
\begin{rem}
We can make similar computations for the $2 \times 3 \times 4$ case without reduced first slice, and the  $2 \times 4 \times 5$ case. The general $2 \times 3 \times 4$ hyperdeterminant has more than $100000$ monomials, and the reduced $2 \times 4 \times 5$ hyperdeterminant has $11912$ monomials with coefficients drawn from $\{\pm 1, \pm 2, \cdots, \pm 8 \}$. We will not print either here.
\end{rem}
\section{Some Open Questions}
We end with a few questions we believe to be open.
\begin{enumerate}
\item It is easy to see that the nondegenerate $3 \times k \times (k+2)$ hypermatrices over $\mathbb{C}$ do not lie in any finite number of orbits. Is there a natural larger group acting on them which rectifies this situation?
\item Is there a simple formula counting nondegenerate $3 \times k \times (k+2)$ hypermatrices? How about for larger formats?
\item Is there some combinatorial interpretation for the terms of the hyperdeterminant, similar to the interpretation of the ordinary determinant in terms of signed permutations?
\end{enumerate}
\section{Acknowledgments}
This research was part of the 2015 summer REU program at the University of Minnesota, Twin Cities, and was supported by RTG grant NSF/DMS-1148634. I would like to thank Joel Lewis and Elise DelMas for their mentorship and valuable advice and comments.

\bibliography{Hyperdeterminants
}{}
\bibliographystyle{plain}

\end{document}